\documentclass[12pt]{article}
\usepackage[utf8]{inputenc}
\usepackage{amsmath}
\usepackage{mathtools}
\usepackage[english]{babel}
\usepackage[T1]{fontenc}
\usepackage{color}

\usepackage[a4paper,top=3cm,bottom=2cm,left=3cm,right=3cm,marginparwidth=1.75cm]{geometry}

\usepackage{amsmath}
\usepackage{graphicx}
\usepackage[colorinlistoftodos]{todonotes}
\usepackage[colorlinks=true, allcolors=blue]{hyperref}

\usepackage{ifpdf}
\ifpdf

\allowdisplaybreaks
\numberwithin{equation}{section}
\usepackage{amsthm}
\newtheorem{theorem}{Theorem}[section]

\theoremstyle{definition}

\theoremstyle{remark}
\newtheorem{remark}[theorem]{Remark}

\begin{document}
\centerline{\bf \large Behavior of Solutions to An Initial Boundary Value Problem}
\centerline{\bf\large for a Hyperbolic System With Relaxation} 
\vskip 0.2cm
\centerline{ CEN Luyu, LIN Lu, LIU Jiyuan, XIAO Yujie}
\centerline{City University of Hong Kong}
\vskip 0.2cm
{\bf Abstract}: The behavior of solutions to an initial boundary value problem 
for a hyperbolic system with relaxation is studied when the relaxation parameter is small, by using the method of
 Fourier Series and the  energy method. 
 \section{Introduction}
 
 In this paper, we consider the initial boundary value problem for the following hyperbolic system with relaxation: 
\begin{equation} \label{eq:a}
    \begin{array}{ccl}
    u_{t}+v_{x}=0\\
    v_{t}+a^2u_{x}=\frac{bu-v}{\epsilon}\\
    \end{array}
\end{equation}
with the boundary conditions:
\begin{equation}
     u(0,t) = u(1,t) = 0, 
\end{equation}
and the initial conditions:
\begin{equation} \label{initial:1}
     u(x,0) = f(x)
\end{equation}
\begin{equation} \label{initial:2}
     v(x,0) = g(x)
\end{equation}
where $a$ and $b$ are constants and $a>0$. $\epsilon>0$ is the relaxation parameter.  We assume that 
\begin{equation}\label{subcharacteristic}
|b|<a, 
\end {equation}
such that the so called ``subcharacteristic condition'' is satisfied.  
System \eqref{eq:a} can be transformed to :
\begin{equation} \label{eq:1}
     u_{tt}-a^2u_{xx}+\frac{1}{\epsilon}
    \left(bu_{x}+u_{t}\right) = 0
\end{equation}
with the  boundary conditions: 
\begin{equation}\label{bc}
     u(0,t) = u(1,t) = 0, 
\end{equation}
and the initial conditions:
\begin{equation} \label{initial}
     u(x,0) = f(x), \ u_t(x, 0)=-g'(x).
\end{equation}
In order to match the initial and boundary conditions, we require that
\begin{equation}\label{match}
f(0)=f(1),  g'(0)=g'(1)=0.\end{equation}

Formally, as $\epsilon\to 0$, system \eqref{eq:a} is relaxed to the equilibrium 
\begin{equation}\label{equi}
u_t+bu_x=0,\end{equation}
\begin{equation}\label{equi1}
v=bu.\end{equation}

System \eqref{eq:a} is related to a general relaxation system of Jin-Xin model \cite{jin1995relaxation}, for which the asymptotic behavior as the relaxation parameter tends to zero of solutions to the initial value problem was discussed in \cite{bianchini2006hyperbolic, natalini1996convergence, tadmor1999pointwise, teng1998first}. The asymptotic behavior as the relaxation parameter tends to zero for the initial boundary problems  in a quarter plane in $(x, t)$ was discussed in \cite{wang1998asymptotic, xin2000stiff, xin2002initial} with one boundary $x=0$. In this paper, we are interested in the behavior of solutions in a $(x, t)$ -strip, $0\le x\le 1$, $t\ge 0$. It should be noted that the Fourier-Laplace transformation is used in \cite{xin2000stiff} to study the problem in a quarter plane. For the 
problem in a strip studied in this paper, we have two boundaries, $x=0$ and $x=1$. This is the main difference of the problem studied in this paper, compared with those in a quarter plane. For example, the  Fourier-Laplace transformation is not applicable to our problem any more. Instead, we use the Fourier series method. \\

The boundary layer behavior is the main concern of this paper, which is shown by the Fourier series solution. We also give the numerical 
simulation based on our Fourier series solution, and justify the zero relaxation limit by the asymptotic expansion with the boundary layer equations and the energy method. \\

Relaxation phenomena are important in many physical situations. For more background, please refer to \cite{chen1994hyperbolic, liu1987hyperbolic}.

\section{Fourier Series Solution.} 
\subsection{Solution Formula}
\begin{theorem}\label{thm2.1} The Fourier series  solution $u(x, t)=\sum_{n=1}^{\infty}T_n(t)X_n(x)$ of problem \eqref{eq:1}, \eqref{bc} and \eqref{initial} is given as follows : \\
i) If  \( \sqrt{\frac{a^2-b^2}{4\epsilon^2a^4\pi^2}}\) is not an integer,
\begin{align} \label{u}
u(x,t) =&\sum_{n=1}^{k}e^{\frac{b}{2a^2\epsilon}x}\sin(n\pi x)\left(c_n e^{\alpha_{n-}t} + d_n e^{\alpha_{n+}t}\right)\notag \\ 
&+ \sum_{n=k+1}^{\infty}e^{\frac{b}{2a^2\epsilon}x-\frac{1}{2\epsilon}t}\sin(n\pi x)
\left[ c_n \cos(\beta_n t)+ d_n \sin(\beta_n t)\right]
\end{align}
where \(k=\lfloor\sqrt{\frac{a^2-b^2}{4\epsilon^2a^4\pi^2}}\rfloor \), 
\begin{align}
c_n&= \int_{0}^{1} f(x)e^{ - \frac{b}{2a^2\epsilon}x}\sin(n\pi x)dx -
          \frac{2\epsilon}{\sqrt{\left(1 - \frac{b^2}{a^2} \right)-4a^2n^2\epsilon^2\pi^2}}
          \int_{0}^{1}\left(\frac{f(x)}{2\epsilon} - g'(x)\right) e^{- \frac{b}{2a^2\epsilon}x} \sin(n\pi x)dx\label{cn1},\\
d_n&= \int_{0}^{1} f(x)e^{ - \frac{b}{2a^2\epsilon}x}\sin(n\pi x)dx +
          \frac{2\epsilon}{\sqrt{\left(1 - \frac{b^2}{a^2} \right)-4a^2n^2\epsilon^2\pi^2}}
          \int_{0}^{1}\left(\frac{f(x)}{2\epsilon} - g'(x)\right) e^{- \frac{b}{2a^2\epsilon}x} \sin(n\pi x)dx\label{dn1},
\end{align}
\begin{equation}\label{alpha}
\alpha_{n\pm}=\frac{-1\pm\sqrt{\left(1-\frac{b^2}{a^2}\right)-4a^2n^2\epsilon^2 \pi^2}}
{2\epsilon},
\end{equation}
for $n \leq  k$, and 
\begin{align}
c_n&= 2\int_{0}^{1} f(x)e^{ - \frac{b}{2a^2\epsilon}x}\sin(n\pi x)dx \label{cn2}\\
d_n&= \frac{4\epsilon}{\sqrt{\left( \frac{b^2}{a^2} -1\right) + 4a^2n^2\epsilon^2\pi^2}}
          \int_{0}^{1}\left(\frac{f(x)}{2\epsilon} - g'(x)\right) e^{- \frac{b}{2a^2\epsilon}x} \sin(n\pi x)dx\label{dn2},
\end{align} 
\begin{equation}
\beta_n = \frac{\sqrt{\left(\frac{b^2}{a^2}-1\right)+4a^2n^2\epsilon^2\pi^2}}{2\epsilon}
\end{equation}
for $n\geq k+1$.\\
ii) If  \( \sqrt{\frac{a^2-b^2}{4\epsilon^2a^4\pi^2}}\) is an integer, let \(k=\sqrt{\frac{a^2-b^2}{4\epsilon^2a^4\pi^2}}.\) Then
\begin{align} 
u(x,t) &=\sum_{n=1}^{k-1}e^{\frac{b}{2a^2\epsilon}x}\sin(n\pi x)\left(c_n e^{\alpha_{n-}t}+d_n e^{\alpha_{n+}t}\right)\notag\\
&+ e^{\frac{b}{2a^2\epsilon}x-\frac{1}{2\epsilon}t}\sin(k\pi x)\left[2\int_{0}^{1}f(x)e^{\frac{-b}{2a^2\epsilon}x}\sin(k\pi x)dx+2t\int_{0}^{1}\left(\frac{f(x)}{\epsilon}-g'(x)\right)e^{2a^2\epsilon}x \sin(k\pi x)dx  \right]\notag\\
&+  \sum_{n=k+1}^{\infty}e^{\frac{b}{2a^2\epsilon}x-\frac{1}{2\epsilon}t}\sin(n\pi x)\left[c_n \cos (\beta_n t) + d_n\sin(\beta_n t)\right]
    \end{align} 
where \(c_n, d_n\) and $\beta_n$ are the same as in case i).

\end{theorem}

\noindent{\it Proof of Theorem \ref{thm2.1}}\\
 For the initial boundary value problem \eqref{eq:1}, \eqref{bc} and \eqref{initial},  we use the separation of variables and let 
\begin{equation} \label{eq:2}
     u(x,t) = X(x)T(t).
\end{equation}
Substitute this in  \eqref{eq:1} to get
\begin{displaymath} 
     T''(t)X(x)-a^2T(t)X''(x)+\frac{1}{\epsilon}
     \left(b T(t)X'(x)+T'(t)X(x)\right) = 0.
\end{displaymath}
Thus,
\begin{displaymath} 
    \frac{T''(t)}{T(t)}-a^2\frac{X''(x)}{X(x)}+\frac{b}{\epsilon}\frac{X'(x)}{X(x)}+\frac{1}{\epsilon}\frac{T'(t)}{T(t)}=0.
\end{displaymath}
Reorganize this and let
\begin{equation} \label{eq:3}
    \frac{T''(t)}{T(t)}+\frac{1}{\epsilon}\frac{T'(t)}{T(t)}=
    a^2\frac{X''(x)}{X(x)}-\frac{b}{\epsilon}\frac{X'(x)}{X(x)}=
    \lambda
\end{equation}
where \(\lambda\) is a constant.

From (\ref{eq:3}), we get ordinary differential equations for $T(t)$ and $X(x)$:
\begin{align}
    T''(t)+\frac{1}{\epsilon}T'(t)-\lambda T(t)=0\label{eq:5}\\
    a^2 X''(x)-\frac{b}{\epsilon}X'(x)-\lambda X(x)=0\label{eq:4}
\end{align}
Moreover, $X(x)$ satisfies boundary conditions
\begin{equation}\label{BCX}
X(0)=X(1)=0. 
\end{equation}
The characteristic equation for  (\ref{eq:4}) is
\begin{equation} 
    a^2\alpha^2-\frac{b}{\epsilon}\alpha-\lambda=0. 
    \label{charateristic}
\end{equation}
Let \(\Delta_1=\left(-\frac{b}{\epsilon}\right)^2+4a^2\lambda\), we have the following cases: \\
Case 1. \(\Delta_1>0\)
\begin{displaymath}
\lambda > -\frac{b^2}{4a^2\epsilon^2}
\end{displaymath}
Then (\ref{charateristic}) has two roots $\alpha_{\pm}$ given by (\ref{alpha}), and 
\begin{displaymath}
X(x)=C_1 e^{\alpha_{-}x}+C_2 e^{\alpha_{+}x}, 
\end{displaymath}
where $C_1$ and $C_2$ are constants. By \eqref{BCX}, we have
\begin{displaymath}
X(0)=C_1+C_2=0, X(1)= C_1e^{\alpha_{-}}+C_2e^{\alpha_{+}}=0
\end{displaymath}
\(\Rightarrow C_1=C_2=0.\)\\
Case 2.
\(\Delta_1=0\)
\begin{displaymath}
\left(-\frac{b}{\epsilon}\right)^2+4a^2\lambda =0
\end{displaymath}
\begin{displaymath}
\lambda = -\frac{b^2}{4a^2\epsilon^2}
\end{displaymath}
Then (\ref{charateristic}) has only one root
\begin{displaymath}
\alpha=\frac{b}{2a^2\epsilon}
\end{displaymath}
\begin{displaymath}
X(x)=C_1e^{\frac{b}{2a^2\epsilon}x}+C_2x
e^{\frac{b}{2a^2\epsilon}x},
\end{displaymath}
for some constants $C_1$ and $C_2$. 
As in Case 1, we get \(C_1=C_2=0.\)\\ 
Case 3. \(\Delta_1<0\)
\begin{displaymath}
\left(-\frac{b}{\epsilon}\right)^2+4a^2\lambda <0
\end{displaymath}
\begin{displaymath}
\lambda < -\frac{b^2}{4a^2\epsilon^2}
\end{displaymath}
Then (\ref{charateristic}) has two complex roots
\begin{displaymath}
\alpha_=\frac{\frac{b}{\epsilon}\pm i\sqrt{-\left(\frac{b}{\epsilon}\right)^2-4a^2\lambda}}
{2a^2}
\end{displaymath}
\begin{displaymath}
X(x)=e^{\frac{b}{2a^2\epsilon}x}
\left(C_1\cos(\frac{\sqrt{-\left(\frac{b}{\epsilon}\right)^2-4a^2\lambda}}{2a^2}x)+C_2\sin(\frac{\sqrt{-\left(\frac{b}{\epsilon}\right)^2-4a^2\lambda}}{2a^2}x)
\right)
\end{displaymath}
\begin{displaymath}
X(0)=0 \Rightarrow C_1=0
\end{displaymath}
\begin{displaymath}
X(1)=C_2e^{\frac{b}{2a^2\epsilon}}
\left(\sin(\frac{\sqrt{-\left(\frac{b}{\epsilon}\right)^2-4a^2\lambda}}{2a^2})\right)=0 
\end{displaymath}

\begin{math}
\Rightarrow \sin(\frac{\sqrt{-\left(\frac{b}{\epsilon}\right)^2-4a^2\lambda}}{2a^2})=0
\Rightarrow \frac{\sqrt{-\left(\frac{b}{\epsilon}\right)^2-4a^2\lambda}}{2a^2}=n\pi,   n=1,2,3,...
\end{math}

\(\Rightarrow \lambda_{n}=-a^2n^2\pi^2-\frac{b^2}{4a^2\epsilon^2}.\)

Therefore, only Case 3 fits the condition, so
\begin{equation}
    X_{n}=e^{\frac{b}{2a^2\epsilon}x}\sin(n\pi x)    
\end{equation}

Next, solve (\ref{eq:5}) with \(\lambda_{n}=-a^2n^2\pi^2-\frac{b^2}{4a^2\epsilon^2}\) 
\begin{displaymath}
T^{''}(t)+\frac{1}{\epsilon}T'(t)+
\left(a^2n^2\pi^2+\frac{b^2}{4a^2\epsilon^2}\right) T(t)=0
\end{displaymath}
Its characteristic equation is:
\begin{equation} 
    \alpha^2+\frac{1}{\epsilon}\alpha+
    \left(a^2n^2\pi^2+\frac{b^2}{4a^2\epsilon^2}\right)=0
    \label{charateristic:2}
\end{equation}
\begin{displaymath}
\Delta_2=\frac{1}{\epsilon^2}\left(1-\frac{b^2}{a^2}\right)-4a^2n^2\pi^2
\end{displaymath}\\
\noindent If \(\Delta_2>0\), two roots of (\ref{charateristic:2}) are
\begin{displaymath}
\alpha_{n\pm}=\frac{-1\pm\sqrt{\left(1-\frac{b^2}{a^2}\right)-4a^2n^2\epsilon^2 \pi^2}}
{2\epsilon}
\end{displaymath}
\begin{equation} \label{2f}
T_{n}(t)=c_n e^{\alpha_{n-}t} + d_n e^{\alpha_{n+}t}
\end{equation}

\noindent If \(\Delta_2<0\), let $\beta_n=\frac{\sqrt{\left( \frac{b^2}{a^2} -1\right)+4a^2n^2\epsilon^2\pi^2}}{2\epsilon}$.
\begin{equation} \label{1f}
\begin{split}
T_{n}(t)=e^{-\frac{1}{2\epsilon}}[c_n \cos(\beta_n t)
+d_n \sin(\beta_n t)]
\end{split}    
\end{equation}\\
\noindent If \(\Delta_2=0\), 
\begin{equation}
T_{n}(t)=c_n e^{-\frac{1}{2\epsilon}t} + d_n t  e^{-\frac{1}{2\epsilon}t}\label{3f}
\end{equation}
We have the following cases: \\

i). \( \sqrt{\frac{a^2-b^2}{4\epsilon^2a^4\pi^2}}\) is not an integer.\\ 
When \(n\leq\lfloor\sqrt{\frac{a^2-b^2}{4\epsilon^2a^4\pi^2}}\rfloor \), \(\Delta_2>0\).
\begin{displaymath}
T_{n}(t)=(\ref{2f})
\end{displaymath}
When \(n>\lfloor\sqrt{\frac{a^2-b^2}{4\epsilon^2a^4\pi^2}}\rfloor \), \(\Delta_2<0\).
\begin{displaymath}
T_{n}(t)=(\ref{1f})
\end{displaymath}

ii). \( \sqrt{\frac{a^2-b^2}{4\epsilon^2a^4\pi^2}}\) is an integer.\\
When \(n<\sqrt{\frac{a^2-b^2}{4\epsilon^2a^4\pi^2}}\), \(\Delta_2>0\).
\begin{displaymath}
T_{n}(t)=(\ref{2f})
\end{displaymath}
When \(n=\sqrt{\frac{a^2-b^2}{4\epsilon^2a^4\pi^2}}\), \(\Delta_2=0\).
\begin{displaymath}
T_{n}(t)=(\ref{3f})
\end{displaymath}
When \(n>\sqrt{\frac{a^2-b^2}{4\epsilon^2a^4\pi^2}}\), \(\Delta_2<0\).
\begin{displaymath}
T_{n}(t)=(\ref{1f})
\end{displaymath}

For case i), let \(k=\lfloor\sqrt{\frac{a^2-b^2}{4\epsilon^2a^4\pi^2}}\rfloor \)
\begin{displaymath}
u(x,t)=\sum_{n=1}^{\infty}X_{n}(x)T_{n}(t)
\end{displaymath}\\
With initial condition (\ref{initial:1}),
\begin{displaymath}
\begin{split}
u(x,0) 
&= e^{\frac{b}{2a^2\epsilon}x}\sum_{n=1}^{k}(c_n + d_n) \sin(n\pi x) +  e^{\frac{b}{2a^2\epsilon}x}\sum_{n=k+1}^{\infty}c_n \sin(n\pi x)= f(x),    
\end{split}
\end{displaymath}
we have\\
\begin{equation}\label{coef1}
c_n + d_n =2\int_{0}^{1}\frac{f(x)}{e^{\frac{b}{2a^2\epsilon}x}}\sin(n\pi x)dx,\quad n=1,2,3,...,k
\end{equation}

\begin{equation}\label{coef2}
    c_n =2\int_{0}^{1}\frac{f(x)}{e^{\frac{b}{2a^2\epsilon}x}}\sin(n\pi x)dx,\quad n=k+1,k+2,k+3,...
\end{equation}\\
Replace $\cos(\beta_n t)$ and $\sin(\beta_n t)$ in (\ref{1f}) according to the following identities
\begin{align} 
\cos(\beta_n t) = \frac{e^{i\beta_n t} + e^{-i\beta_n t}}{2}\\
\sin(\beta_n t) = \frac{e^{i\beta_n t} - e^{-i\beta_n t}}{2i}
\end{align}

$u(x,t)$ can be written as
\begin{equation}
\begin{split}
u(x,t) &= e^{\frac{b}{2a^2\epsilon}x}\sum_{n=1}^k (c_n e^{\alpha_{n-} t} + d_n e^{\alpha_{n+}t}) \sin(n\pi x) \\
&+  e^{\frac{b}{2a^2\epsilon}x}\sum_{n=k+1}^{\infty}\left[\frac{c_n - id_n}{2}e^{i\beta_n t} + \frac{c_n + id_n}{2}e^{-i\beta_n t}\right]e^{-\frac{1}{2\epsilon}t}
\end{split}
\end{equation}
With initial condition (\ref{initial:2}),
\begin{equation}
\begin{split}
u_{t}(x,0) &= e^{\frac{b}{2a^2\epsilon}x}\sum_{n=1}^k (c_n \alpha_{n-}+ d_n \alpha_{n+}) \sin(n\pi x) \\
&+  e^{\frac{b}{2a^2\epsilon}x}\sum_{n=k+1}^{\infty}\left[\frac{c_n - id_n}{2}(i\beta_n - \frac{1}{2\epsilon}) + \frac{c_n + id_n}{2}(-i\beta_n - \frac{1}{2\epsilon})\right]\\
                 & = -g'(x),
\end{split}
\end{equation}
 
we have
\begin{equation}\label{coef3}
 c_n \alpha_{n-} + d_n\alpha_{n+} = 2\int_{0}^{1}-g'(x){e^{-\frac{b}{2a^2\epsilon}x}}\sin(n\pi x)dx, \quad n\leq k\\
\end{equation}
and
\begin{align}
  \frac{c_n - id_n}{2}(i\beta_n - \frac{1}{2\epsilon}) &+ \frac{c_n + id_n}{2}(-i\beta_n - \frac{1}{2\epsilon})\notag\\
 &=2\int_{0}^{1}-g'(x){e^{-\frac{b}{2a^2\epsilon}x}}\sin(n\pi x)dx n\ge k+1,\end{align}
 which implies 
\begin{equation}\label{coef4}
-\frac{1}{2\epsilon}c_n + d_n \beta_n 
=2\int_{0}^{1}-g'(x){e^{-\frac{b}{2a^2\epsilon}x}}\sin(n\pi x)dx, n\ge k+1.
\end{equation}

When $n \leq  k$, we obtain $c_n$ and $d_n$ given by \eqref{cn1} and \eqref{dn1}  from (\ref{coef1}) and (\ref{coef3}).  
When $n\geq k+1$,  the formula for $c_n$ and $d_n$ follow from (\ref{coef2}) and (\ref{coef4}), and \eqref{u} is proved.

Case ii) in Theorem \ref{thm2.1} can be shown similarly.
    
\subsection{Analysis of the Solutions of Fourier Series }
We prove the following Theorem
\begin{theorem}\label{thm2.2} For the solution given in Theorem \ref{thm2.1}, let $u_n(x, t)=T_n(t)X_n(x)=:A_n(x, t)sin(n\pi x)$. Then when $\epsilon$ is sufficiently small, \\i) for $b>0$,
\begin{equation}\label{eu1}
|A_1(x, t)|\le A\epsilon \exp\left(\frac{b}{2a^2\epsilon}\left(x-\frac{a^2}{b}\left(1-\sqrt{1-\frac{b^2}{a^2}-4a^2\epsilon^2 \pi^2}\right)t\right)\right),
\end{equation} 
for small $\epsilon$. 
\begin{equation}\label{eu1a}
|A_{k-m}(x, t)|\le\frac{ B\sqrt{\epsilon}}{\sqrt m} \exp\left(\frac{b}{2a^2\epsilon}\left(x-\frac{a^2}{b}\left(1-\sqrt{1-\frac{b^2}{a^2}-4a^2\epsilon^2 (k-m)^2\pi^2}\right)t\right)\right),
\end{equation}
\begin{equation}\label{eu1b}
|A_{k+m}(x, t)|\le\frac{ C\sqrt{\epsilon}}{\sqrt m} \exp\left(\frac{b}{2a^2\epsilon}\left(x-\frac{a^2}{b}t\right)\right),
\end{equation}
for \(k=\lfloor\sqrt{\frac{a^2-b^2}{4\epsilon^2a^4\pi^2}}\rfloor \) , $m\ge 1$, and $\epsilon\le \frac{\delta}{m}$ for some small $\delta>0$, where $A$,  $B$ and $C$  are constants independent of $\epsilon$, defined in $0\le x\le 1$ and $t>0$.
\\ii) for $b=0$, $t>0$,\\
 for any fixed $n<k$, $k=\lfloor\frac{1}{2\epsilon a^2 \pi^2}\rfloor$, $ A_n(x, t)-2\int_{0}^1 f(x)\sin(nx) dx\to 0$ as $\epsilon \to 0$; \\
 for $n=k+m$, $m\geq1$, $A_n(x,t)\to 0$ as $\epsilon \to 0$.
\end{theorem}
\begin{remark} The case for $b<0$ can be discussed as is the case for $b>0$, by replacing $x$ by $1-x$.\end{remark}
\begin{remark} $A_n$ ($n\ge 1$) are the amplitutes of Fourier modes. For $b>0$, by \eqref{eu1}, and  \eqref{eu1a}, we have, for $n<k$, that $A_n(x, t)\to 0$ as $\epsilon\to 0$ for  $\left(x-\frac{a^2}{b}\left(1-\sqrt{1-\frac{b^2}{a^2}}\right)t\right)<0$. For $n>k$, we have that $A_n(x, t)\to 0$ as $\epsilon\to 0$ for  $\left(x-\frac{a^2}{b}t\right)<0$. \end{remark}
\noindent{\it Proof of Theorem \ref{thm2.2}}. 
\\i)\\
For $n=1$, we note that
\begin{equation}\label{alpha1}
\alpha_{1\pm}=\frac{-1\pm\sqrt{\left(1-\frac{b^2}{a^2}\right)-4a^2\epsilon^2 \pi^2}}
{2\epsilon},
\end{equation}
\begin{align}
c_1&= \int_{0}^{1} f(x)e^{ - \frac{b}{2a^2\epsilon}x}\sin(\pi x)dx -
          \frac{2\epsilon}{\sqrt{\left(1 - \frac{b^2}{a^2} \right)-4a^2\epsilon^2\pi^2}}
          \int_{0}^{1}\left(\frac{f(x)}{2\epsilon} - g'(x)\right) e^{- \frac{b}{2a^2\epsilon}x} \sin(\pi x)dx\label{c1},\\
d_1&= \int_{0}^{1} f(x)e^{ - \frac{b}{2a^2\epsilon}x}\sin(\pi x)dx +
          \frac{2\epsilon}{\sqrt{\left(1 - \frac{b^2}{a^2} \right)-4a^2\epsilon^2\pi^2}}
          \int_{0}^{1}\left(\frac{f(x)}{2\epsilon} - g'(x)\right) e^{- \frac{b}{2a^2\epsilon}x} \sin(\pi x)dx\label{d11},
\end{align}
\begin{equation}\label{u1}
u_1(x, t)=e^{\frac{b}{2a^2\epsilon}x}\sin(\pi x)\left(c_1 e^{\alpha_{1-}t} + d_1 e^{\alpha_{1+}t}\right)=A_1(x, t)\sin(\pi x).
\end{equation}
Apparently, \begin{equation}\label{u1a}
|A_1|(x, t)\le (|c_1|+|d_1|)e^{\frac{b}{2a^2\epsilon}x+\alpha_{1+}t}, 
\end{equation}
for $0\le x\le 1$ and $t>0$. We estimate $c_1$ and $d_1$ as follows, 
\begin{align}\label{ab1}
&|\int_{0}^{1} f(x)e^{ - \frac{b}{2a^2\epsilon}x}\sin(\pi x)dx|\notag\\
\le &\max_{x\in [0, 1]}|f(x)|\int_{0}^{1} e^{ - \frac{b}{2a^2\epsilon}x}dx\le \frac{2a^2\epsilon}{b} \max_{x\in [0, 1]}|f(x)|, \end{align}
\begin{align}\label{ab2}
&\left|\frac{2\epsilon}{\sqrt{(1 - \frac{b^2}{a^2} )-4a^2\epsilon^2\pi^2}}
          \int_{0}^{1}\frac{f(x)}{2\epsilon} e^{- \frac{b}{2a^2\epsilon}x} \sin(\pi x)dx\right|\notag\\
 &\le \frac{\sqrt 2}{\sqrt{1 - \frac{b^2}{a^2}}} \int_{0}^{1}|f(x)| e^{- \frac{b}{2a^2\epsilon}x}dx\notag\\
 &\le  \frac{\sqrt 2}{\sqrt{1 - \frac{b^2}{a^2} }}\frac{2a^2\epsilon}{b} \max_{x\in [0, 1]}|f(x)|, 
 \end{align}
 \begin{align}\label{ab2}
&\left|\frac{2\epsilon}{\sqrt{(1 - \frac{b^2}{a^2} )-4a^2\epsilon^2\pi^2}}
          \int_{0}^{1}g'(x)e^{- \frac{b}{2a^2\epsilon}x} \sin(\pi x)dx\right|\notag\\
&\le  \frac{\sqrt 2\epsilon}{\sqrt{1 - \frac{b^2}{a^2} }}\frac{2a^2\epsilon}{b} \max_{x\in [0, 1]}|g'(x)|, 
 \end{align}
 for small $\epsilon$. Hence\\
$$|c_{1}|+|d_{1}|\le  \frac{4a^2\epsilon}{b} \max_{x\in [0, 1]}|f(x)|+\frac{4\sqrt {2}}{\sqrt{1-\frac{b^2}{a^2}}}\frac{a^2\epsilon}{b} \max_{x\in [0, 1]}(|f(x)|+\epsilon|g'(x)|).$$

(\ref{eu1}) follows from the above estimates then.\\

For $n=k-m$, $m\ge 1$, we estimate $c_{k-m}$ and $d_{k-m}$ as follows. 
\begin{equation*}
\begin{split}
& 1 - \frac{b^2}{a^2} -4a^2(k-m)^2\epsilon^2\pi^2\\
&=1-\frac{b^2}{a^2}-4a^2\epsilon^2\pi^2(\lfloor\frac{\sqrt{a^2-b^2}}{2a^2\pi\epsilon}\rfloor^2-2\lfloor\frac{\sqrt{a^2-b^2}}{2a^2\pi\epsilon}\rfloor m+m^2)\\
&\geq 8a^2\epsilon^2\pi^2m\lfloor\frac{\sqrt{a^2-b^2}}{2a^2\pi\epsilon}\rfloor-4a^2\epsilon^2\pi^2m^2\\
&\geq 2\epsilon\pi m \sqrt{a^2-b^2}-4a^2\epsilon^2\pi^2m^2\\
&\geq \pi m\epsilon \sqrt {a^2-b^2},
\end{split}
\end{equation*}
if $\epsilon\le \frac{\sqrt {a^2-b^2}}{4\pi a^2 m}$. \\
Hence
$$|c_{k-m}|+|d_{k-m}|\le  \frac{4a^2\epsilon}{b} \max_{x\in [0, 1]}|f(x)|+\frac{4\sqrt \epsilon}{\sqrt{\pi m \sqrt {a^2-b^2}}}\frac{2a^2}{b} \max_{x\in [0, 1]}(|f(x)|+\epsilon|g'(x)|),$$
if $\epsilon\le \frac{\delta}{m}$ for some small $\delta$.  This proves \eqref{eu1a}. \eqref{eu1b} can be proved similarly. 
\\
ii)\\
When $b = 0$,  the equilibrium equation \eqref{equi} becomes $u_t=0$ to which we denote the solution by  $\bar{u}(x, t)$ which is independent of $t$. Then we have 
\begin{equation*}
    \bar{u}(x, t) = f(x) = \sum_{n=0}^{\infty}a_n \sin(n\pi x), \quad \text{where } a_n = 2\int_{0}^1 f(x)\sin(n\pi x) dx.
\end{equation*}
Denote the solution for the problem (\ref{eq:1}), \eqref{bc} and \eqref{initial} by $u(x, t)=\sum_{n=1}^{\infty} A_n(x, t) \sin(n\pi x)$. 
We will show that for any fixed $n<k$, $ A_n(x, t)-a_n\to 0$ as $\epsilon \to 0$ for $t>0$. \\
From (\ref{u}), when $b = 0$, we have $k=\lfloor\frac{1}{2\epsilon a\pi}\rfloor$ ,
\begin{align} 
u(x,t) =&\sum_{n=1}^{k}\sin(n\pi x)\left(c_n e^{\alpha_{n-}t} + d_n e^{\alpha_{n+}t}\right)\notag \\ 
&+ \sum_{n=k+1}^{\infty}\sin(n\pi x)e^{-\frac{1}{2\epsilon}t}
\left[c_n \cos(\beta_n t)+ d_n \sin(\beta_n t)\right]
\end{align}
where, for $n\le k$, 
\begin{displaymath}
\alpha_{n\pm}=\frac{-1\pm\sqrt{1-4a^2n^2\epsilon^2 \pi^2}}
{2\epsilon},
\end{displaymath}
\begin{align}
c_n&= \int_{0}^{1} f(x)\sin(n\pi x)dx -
          \frac{2\epsilon}{\sqrt{1-4a^2n^2\epsilon^2\pi^2}}
          \int_{0}^{1}\left(\frac{f(x)}{2\epsilon} - g'(x)\right) \sin(n\pi x)dx,\\
d_n&= \int_{0}^{1} f(x)\sin(n\pi x)dx +
          \frac{2\epsilon}{\sqrt{1 - 4a^2n^2\epsilon^2\pi^2}}
          \int_{0}^{1}\left(\frac{f(x)}{2\epsilon} - g'(x)\right) \sin(n\pi x)dx
\end{align}
and  for $n> k$,
\begin{align}
c_n&= 2\int_{0}^{1} f(x)\sin(n\pi x)dx, \\
d_n&= \frac{4\epsilon}{\sqrt{4a^2n^2\epsilon^2\pi^2-1}}
          \int_{0}^{1}\left(\frac{f(x)}{2\epsilon} - g'(x)\right)  \sin(n\pi x)dx.
\end{align}

Obviously, for any fixed $n<k$ and $t>0$, $e^{\alpha_n- t}\to 0$ as $\epsilon\to 0$.

We denote $$u(x, t)=\sum_{n=1}^{\infty} A_n(x, t) \sin(n\pi x). $$

When $n \leq k$,
let $w_n = \frac{4\epsilon}{\sqrt{1-4a^2n^2\epsilon^2\pi^2}}
          \int_{0}^{1}\left(\frac{f(x)}{2\epsilon} - g'(x)\right)  \sin(n\pi x)dx $. Then 
\begin{align}
c_n&= \frac{a_n - w_n}{2}\\
d_n&= \frac{a_n + w_n}{2}
\end{align}
We employ Taylor expansion to $\alpha_{n+}$ to get
\begin{equation*}
    \alpha_{n+} = \frac{-1 + \sqrt{1-4a^2 n^2 \epsilon^2 \pi^2}}{2\epsilon} = \frac{-1 + 1 - \frac{4a^2 n^2 \epsilon^2 \pi^2}{2} +n^4 O(\epsilon ^ 4)}{2\epsilon} 
    = -a^2 n^2 \pi^2\epsilon +n^4 O(\epsilon^3)
\end{equation*}

Furthermore we have 
\begin{equation*}
e^{\alpha_{n+}t} = 1 - a^2 n^2 \pi^2\epsilon t +n^4t O(\epsilon^3)
\end{equation*}
and
\begin{align*}
    w_n &= \frac{4\epsilon}{\sqrt{1-4a^2n^2\epsilon^2\pi^2}}
          \int_{0}^{1}\left(\frac{f(x)}{2\epsilon} - g'(x)\right)  \sin(n\pi x)dx \\
        &= \frac{2}{\sqrt{1-4a^2n^2\epsilon^2\pi^2}} \int_{0}^{1}\left(f(x) - 2\epsilon g'(x)\right)\sin(n\pi x)dx
       \end{align*}
       By Taylor expansion, it is easy to show that
       $$w_n=a_n+O(\epsilon)+n^2O(\epsilon^2),$$
       for any fixed $n<k$, as $\epsilon \to 0$.

Therefore, when $\epsilon \rightarrow 0$,
\begin{equation*}\begin{split}
&d_n e^{\alpha_{n+}t} - a_n\\
& =\frac{a_n + w_n}{2}e^{\alpha_{n+}t} - a_n\\
& =
\frac{a_n + a_n + O(\epsilon) + n^2O(\epsilon^2)}{2}(1 - a^2 n^2 \pi^2\epsilon t +n^4 t O(\epsilon^3))-a_n \\
&= O(\epsilon) + n^2(1+t) O(\epsilon^2). 
\end{split}
\end{equation*}
Since $e^{\alpha_{n-}t} \rightarrow 0$ as $\epsilon \rightarrow 0$ for $t>0$ and $n<k$, we have $A_n(x, t)-a_n\to 0$ as $\epsilon \to 0$ for any fixed $n<k$ and $t>0$. The case $n=k+m$ for $m\ge 1$ can be analysed as in the case when $b>0$.

\section{Analysis by the energy method}

\subsection{The case when $b=0$}
In the case of $b=0$, the equilibrium equation \eqref{equi} becomes

\begin{equation*}
\bar{u}_{t}=0.
\end{equation*}
With the initial value $\bar{u}(x, 0)=f(x)$, then we have
$\bar u(x, t)=f(x), \ x\in[0,1], t\ge 0$. 
Let $u$ be the smooth solution of \eqref{eq:1}, \eqref{bc} and \eqref{initial}. Set
\begin{equation}
w=u-\bar{u}
\end{equation}
Then $w$ is a solution to the following initial boundary value problem: 
\begin{equation}\label{w}\begin{cases}
& w_{tt}-a^2w_{xx}-a^2u_{xx}+\frac{1}{\epsilon}w_t=0, 0\le x\le 1, t>0, \\
&w(0, t)=w(1, t)=0, \\
&w(x, 0)=0, w_t(x, 0)=-g'(x).\end{cases}
\end{equation}
\begin{theorem}\label{thm3.1}  Let $w$ be the solution to problem \eqref{w}. It then holds that 
\begin{equation}\label{10171} 
\int_0^1w^2(x,t)dx+\int_0^t \int_0^1 w_t^2(x, s)dxds \leq C\epsilon \left(\int_0^1(g'(x))^2dx+t\int_0^1(f'(x))^2+(f''(x))^2)dx\right), 
\end{equation}
for $0<\epsilon<1/4$ and $t>0$, where $C$ is a constant only depending on $a$. 
\end{theorem}
\begin{proof} Multiply the first equation in \eqref{w} by $w_t$ and $w$ respectively, and integrate the resulting equations by parts over $[0, 1]$, and use the boundary conditions to get
\begin{equation}\label{energy1}
\frac{d}{dt}\int_{0}^{1}(\frac{w^2}{2\epsilon}+w w_{t})(x, t){dx}-\int_{0}^{1}w_t^2(x, t){dx}+a^2 \int_{0}^{1}w_x^2(x, t){dx}=-\int_{0}^{1}a^2\bar{u}_x w_x(x, t) dx\end{equation}
\begin{equation}\label{energy2}
\frac{d}{dt}\int_{0}^{1}(\frac{1}{2}w_t^2+\frac{a^2}{2}w_x^2)(x, t) dx+\frac{1}{\epsilon}\int_{0}^{1}w_t^2(x, t)dx=a^2\int_{0}^{1}\bar{u}_{xx}w_t(x, t)dx. \end{equation}
By Cauchy-Schwarz inequality, we have
\begin{equation*}
\int_0^1|a^2\bar{u}_{xx}w_t|dx \leq \int_0^1(\frac{1}{2}w_t^2+\frac{a^4}{2}\bar u_{xx}^2)dx
\end{equation*}
\begin{equation*}
\int_0^1a^2|\bar{u}_x w_x |dx\leq \int_0^1 (\frac{a^2}{2}w_x^2+\frac{a^2}{2}\bar{u}_x^2)dx.
\end{equation*}
Therefore, 
\begin{equation*}
\begin{split}
&\frac{d}{dt}\int_{0}^{1}(\frac{1}{2\epsilon}w^2+ww_t+\frac{1}{2}w_t^2+\frac{a^2}{2}w_x^2)dx+(\frac{1}{\epsilon}-1)\int_{0}^{1}w_t^2dx+a^2\int_{0}^{1}w_x^2dx\\
&=-\int_0^1a^2\bar{u}_x w_x dx+\int_0^1a^2\bar{u}_{xx}w_t dx\\
& \leq \frac{a^2}{2} \int_0^1 \bar{u}_x^2 dx+ \frac{a^2}{2} \int_0^1 w_x^2 dx+\int_{0}^{1}(\frac{1}{2}w_t^2+\frac{a^4}{2}\bar u_{xx}^2) dx.
\end{split}
\end{equation*}
Hence, 
\begin{equation}\label{energyiequality1}
\begin{split}
&\frac{d}{dt}\int_{0}^{1}(\frac{1}{2\epsilon}w^2+ww_t+\frac{1}{2}w_t^2+\frac{a^2}{2}w_x^2)dx+(\frac{1}{\epsilon}-\frac{3}{2})\int_{0}^{1}w_t^2dx+\frac{a^2}{2}\int_{0}^{1}w_x^2dx\\
& \leq \int_0^1 (\frac{a^2}{2}\bar{u}_x^2+\frac{a^4}{2}\bar u_{xx}^2) dx.
\end{split}
\end{equation}
Note that since
\begin{equation}
    -(w^2+\frac{1}{4}w_t^2)\leq ww_t \leq w^2+\frac{1}{4}w_t^2,
\end{equation}
\begin{equation*}
\frac{1}{2\epsilon}w^2+w w_t+\frac{1}{2}w_t^2 \geq (\frac{1}{2\epsilon}-1)w^2+\frac{1}{4}w_t^2.
\end{equation*}
Integrate \eqref{energyiequality1} with respect to time to get 
\begin{equation*}
\begin{split}
&\int_0^1((\frac{1}{2\epsilon}-1)w^2+\frac{1}{4}w_t^2+\frac{a^2}{2}w_x^2)(x,t)dx +(\frac{1}{\epsilon}-\frac{3}{2})\int_0^t\int_0^1w_t^2(x, s) dx ds+\frac{a^2}{2}\int_0^t\int_0^1w_x^2 (x, s) dxds\\
&\leq\int_0^1((\frac{1}{2\epsilon}+1)w^2+\frac{3}{4}w_t^2+\frac{a^2}{2}w_x^2)(x,0)dx+\int_0^t \int_0^1(\frac{a^2}{2}\bar{u}_x^2+\frac{a^4}{2}\bar{u}_{xx}^2)(x, s)dxds.
\end{split}
\end{equation*} 
By the initial condition $w(x, 0)=w_x(x, 0)=0$ and $w_t(x, 0)=-g'(x)$, we have
\begin{equation*}
\begin{split}
&\int_0^1(\frac{1}{2\epsilon}-1)w^2(x,t)dx+(\frac{1}{\epsilon}-\frac{3}{2})\int_0^t \int_0^1 w_t^2 dxds\\
&\leq \int_0^1 \frac{3}{4}(g'(x))^2(x,0)dx+\int_0^t\int_0^1\frac{a^2}{2}(\bar{u}_x^2+\frac{a^4}{2}\bar{u}_{xx}^2)(x, s)dxds
\end{split}
\end{equation*}
This proves \eqref{10171} and completes the proof of Theorem \ref{thm3.1}. \end{proof}

\subsection{The case for $b<0$} 
\includegraphics[width=\textwidth]{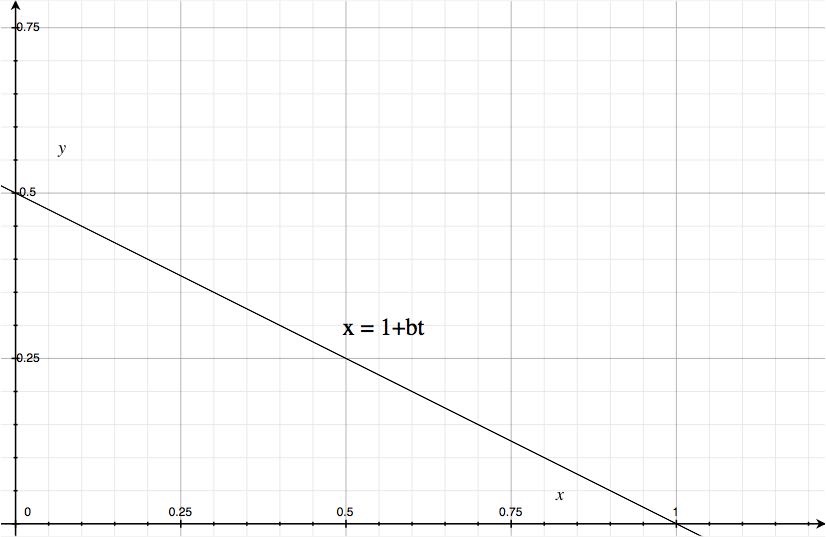}

In this subsection, we present the analysis for the case when $b<0$ by using the boundary layer profile and the energy method. When $b<0$, the boundary layer occurs at the boundary $x=0$. The case for $b>0$ can be handled similarly for which the boundary layer occurs at $x=1$. 
Denote the solution at equilibrium as $u^e(x,t)$, then it satisfies the following equations:
\begin{align}
&\partial_t u^e + b\partial_x u^e = 0, \\
&u^e(1,t) = 0,\\
&u^e(x,0) = f(x). 
\end{align}
Solving for $u^e$, we have:
\begin{equation}\label{ue}
u^e(x,t) =  \Big \{
\begin{aligned}
f(x-bt), \quad&x \leq 1+bt, \\
0, \quad &x > 1+bt.
\end{aligned}
\Big.
\end{equation} 

The solution $u^e$ is illustrated in the above figure. 
Taking $a = 1$ for convenience, we may write the problem \eqref{eq:1}, \eqref{bc} and \eqref{initial} as 
\begin{align}
&u_{tt}^\epsilon - u_{xx}^\epsilon + \frac{1}{\epsilon}(u_{t}^\epsilon+bu_{x}^\epsilon) = 0,\label{uepsilon}\\
&u^\epsilon(0,t) = u^\epsilon(1,t)=0,\\
&u^\epsilon(x,0) = f(x),\\
&u_t^\epsilon(x,0) = -g'(x),
\end{align}
where we have used $u^{\epsilon}$ to indicate the dependence of the solution on $\epsilon$. 

Boundary layer expansion:
\begin{equation}\label{ubnd}
u^\epsilon(x,t) = u^e(x,t) + U_0(y,t) + \epsilon U_1(y,t) + w(x,t),
\end{equation}
where $y = \frac{x}{\epsilon}.$ 

Plug (\ref{ubnd}) in (\ref{uepsilon}) to obtain:
\begin{equation}\label{ubnd2}
\begin{split}
&\partial_t^2 u^e - \partial_x^2 u^e + \frac{1}{\epsilon}(\partial_t u^e + b\partial_x u^e)\\
&+\partial_t^2 U_0 - \frac{1}{\epsilon^2}\partial_y^2 U_0 + \frac{1}{\epsilon}(\partial_t U_0 + \frac{b}{\epsilon}\partial_y U_0)\\
&+\epsilon[\partial_t^2 U_1 - \frac{1}{\epsilon^2}\partial_y^2 U_1 + \frac{1}{\epsilon}(\partial_t U_1 + \frac{b}{\epsilon}\partial_y U_1)]\\
&=\partial_t^2 w - \partial_x^2 w + \frac{1}{\epsilon}(\partial_t w + b\partial_x w) = 0.
\end{split}
\end{equation}

To make $O(\epsilon^{-2})$ and $O(\epsilon^{-1})$ in the above equation be of order 0, we have:
\begin{align}
&O(\epsilon^{-2}): -\partial_y^2 U_0 + b\partial_y U_0 = 0,\label{o2}\\
&O(\epsilon^{-1}): \partial_t U_0 - \partial_y^2 U_1 + b\partial_y U_1 = 0.\label{o1}
\end{align}

After solving for (\ref{o2}) and (\ref{o1}), we take 
\begin{align}
&U_0(y,t) = c(t)e^{by},\label{U0}\\
&U_1(y,t) = \frac{c'(t)}{b}[(y-\frac{1}{b})e^{by}+\frac{1}{b}] + \frac{d(t)}{b}(e^{by}-1).\label{U1}
\end{align}

Then (\ref{ubnd2}) becomes:
\begin{equation}
\begin{split}
\partial_t^2 u^e - \partial_x^2 u^e + \partial_t^2 U_0 + \epsilon \partial_t^2 U_1 
+\partial_t U_1 + \partial_t^2 w - \partial_x^2 w\\ + \frac{1}{\epsilon}(\partial_t w + b\partial_x w) = 0.\label{ubnd3}
\end{split}
\end{equation}
Recall that, from (\ref{ubnd}),
\begin{equation}
w  = u^{\epsilon} - U_0 - \epsilon U_1 - u^e.
\end{equation}
We want $w(0,t) = w(1,t) = 0.$ Since $u^{\epsilon}(0,t) = u^{\epsilon}(1,t) = 0$ and
\begin{equation}
u^e(1,t) = 0,
\end{equation}
\begin{equation}
u^e(0,t) = \Big \{
\begin{aligned}
f(-bt), \quad &t \leq \frac{-1}{b},\\
0, \qquad &t > \frac{-1}{b},
\end{aligned}
\end{equation}

to make w(0,t) = w(1,t) = 0, the following equations must be satisfied:
\begin{align}
&(U_0+\epsilon U_1)(0,t) = -u^e(0,t)\label{biguleft}\\
&(U_0+\epsilon U_1)(1,t) = 0\label{biguright}
\end{align}

Solving for (\ref{biguleft}) and (\ref{biguright}), we have
\begin{equation}
c(t) = \Big \{
\begin{aligned}
-f(-b t), \quad &t \leq \frac{-1}{b}\\
0, \qquad &t > \frac{-1}{b}
\end{aligned}
\end{equation}

\begin{equation}
d(t) = \frac{b c(t)\frac{e^{\frac{b}{\epsilon}}}{\epsilon}+c'(t)[(\frac{1}{\epsilon}-\frac{1}{b})e^{\frac{b}{\epsilon}}+\frac{1}{b}]}{1-e^{\frac{b}{\epsilon}}}\label{dt}
\end{equation}
Also, we can see that 
\begin{equation}
\lim_{\epsilon \rightarrow 0}d(t) = \frac{c'(t)}{b}
\end{equation}

Consider (\ref{ubnd3}). Now we have:

\begin{equation}
\partial_t^2 w - \partial_x^2 w + \frac{1}{\epsilon}(\partial_t w + b\partial_x w) + (\partial_t^2 u^e - \partial_x^2 u^e) + (\epsilon \partial_t^2 U_1 + \partial_t U_1 + \partial_t^2 U_0) = 0 \label{weqn}
\end{equation}

with boundary conditions
\begin{equation}
w(0,t) = w(1,t) = 0\label{wbc}
\end{equation}

and initial conditions
\begin{align}
w(x,0) &= u^{\epsilon}(x,0) - u^e(x,0) - U_0(\frac{x}{\epsilon},0) - \epsilon U_1(\frac{x}{\epsilon},0)\nonumber \\ 
& = -\epsilon U_1(\frac{x}{\epsilon},0),\label{wic}\\
w_t(x, 0) &= u_{t}^{\epsilon} - u_{t}^{e} - U_{1t} - U_{0t}.\label{wtic}
\end{align}
From the definitions of $u^{e}$, $U_0$ and $U_1$ in (\ref{ue}), (\ref{U0}), and (\ref{U1}), it is easy to see that
\begin{align}\label{initialestimates} 
&|w(x, 0)|\le C\epsilon |f'(0)|, \\\notag
 &|w_t(x, 0)|\le C\epsilon (|f'(0)|+|f^{"}(0)|)+C(|g'(x)|+|b||f'(x)|),\\\notag
 &|w_x(x, 0)|\le C|f'(0)|\notag ,\end{align}
 where and in the following, we use $C$ to denote a generic constant independent of $\epsilon$. 
 
 \begin{theorem}\label{mainthm}
 Suppose that $f\in C^3([0,1])$ and $f'(0)=1$. Let $w$ be the solution to problem \eqref{weqn}, \eqref{wbc}, \eqref{wic} and \eqref{wtic}. Then it holds that
\begin{align}\label{wbnd3}
&\int_0^1 w^2(x, t)dx+\int_0^t\int_0^1 w^2(x,s) dxds\notag\\
&\le C\epsilon^2 t\left(\int_0^1((g'(x))^2+b^2|f'(x)|^2)dx+ (|f'(0)|^2+|f^{''}(0)|^2)\right)\notag\\
&+C\epsilon t^2 \max_{[0, 1]}\sum_{i=1}^3|f^{(i)}(x)|^2.
\end{align} 
\end{theorem}
\begin{remark} It is easy to verify that $\int_0^1 U_0^2(\frac{x}{\epsilon}, t)dx\le O(1)\epsilon$ and $\int_0^1\epsilon^2 U_1^2(\frac{x}{\epsilon}, t)dx\le O(1)\epsilon^2. $ It follows from 
\eqref{wbnd3} that, for any fixed $t>0$, $\int_0^1 (u^{\epsilon}-u^e)^2(x, t)dx $ converges to zero in the order of $\epsilon$ as 
$\epsilon\to 0$. \end{remark}

\noindent {\it Proof of Theorem \ref{mainthm}}

Multiply (\ref{weqn}) by $w_t$, then integrate both sides on (0,1) with respect to x:

\begin{multline}
\frac{d}{dt}\int_0^1 (\frac{w_t^2}{2}+\frac{w_x^2}{2})dx + \frac{1}{\epsilon}\int_0^1 w_t^2 dx + \frac{b}{\epsilon}\int_0^1 w_xw_t dx\\
= - \int_0^1 (\partial_t^2 u^e - \partial_x^2 u^e)w_t dx - \int_0^1(\epsilon \partial_t^2 U_1 + \partial_t U_1 + \partial_t^2 U_0)w_tdx.\label{mwt}
\end{multline}

Similarly, multiply (\ref{weqn}) by $w$, then integrate both sides on (0,1) with respect to x:

\begin{multline}
\frac{d}{dt}\int_0^1 (ww_t+\frac{1}{2\epsilon}w^2)dx+\int_0^1(w_x^2-w_t^2)dx\\
=- \int_0^1 (\partial_t^2 u^e - \partial_x^2 u^e)w dx - \int_0^1(\epsilon \partial_t^2 U_1 + \partial_t U_1 + \partial_t^2 U_0)wdx.\label{mw}
\end{multline}

Denote
\begin{equation}
G(x,t) = \left[(\partial_t^2 u^e(x,t)- \partial_x^2 u^e(x,t))\right] + \left[\epsilon \partial_t^2 U_1(\frac{x}{\epsilon},t)+ \partial_t U_1(\frac{x}{\epsilon},t) + \partial_t^2 U_0(\frac{x}{\epsilon},t)\right].
\end{equation}

Then $(\ref{mwt}) + (\ref{mw})\times k$ for a constant $k$ to be determined later  yields that
\begin{multline}
\frac{d}{dt}\int_0^1 (\frac{w_t^2+w_x^2}{2}+kww_t+\frac{kw^2}{2\epsilon})dx
+\int_0^1 (\frac{1}{\epsilon}-k)w_t^2dx
+\frac{b}{\epsilon}\int_0^1 w_xw_t dx + k\int_0^1 w_x^2 dx \\
= - \int_0^1 Gw_t dx - k\int_0^1 Gw dx. \label{emth}
\end{multline}

The left hand side of (\ref{emth}) can be written as :
\begin{align}
L.H.S=&\frac{d}{dt}\int_0^1 \{\frac{1}{2}[(w_t+kw)^2+(\frac{k}{\epsilon}-k^2)w^2]+\frac{1}{2}w_x^2\}dx \nonumber \\
&+ \int_0^1 k(w_x+\frac{b}{2\epsilon k}w_t)^2 + (\frac{1}{\epsilon}-k-\frac{b^2}{4\epsilon^2 k})w_t^2 dx.
\end{align}

In order to have positive coefficients, we have:
\begin{align}
\frac{k}{\epsilon}-k^2 &> 0,\\
\frac{1}{\epsilon}-k-\frac{b^2}{4\epsilon^2 k}&>0,\\
k&>0.
\end{align}

Solving for the above conditions, we obtain:
\begin{equation}
k \in (\frac{1-\sqrt{1-b^2}}{2\epsilon},\frac{1+\sqrt{1-b^2}}{2\epsilon}).
\end{equation}

Thus we take 
\begin{equation}
k = \frac{1}{2\epsilon}.
\end{equation}

Then (\ref{emth}) becomes:
\begin{multline}
\frac{1}{2}\frac{d}{dt}\int_0^1 [(w_t+\frac{w}{2\epsilon})^2+\frac{1}{4\epsilon^2}w^2+w_x^2]dx + \int_0^1 [\frac{1}{2\epsilon}(w_x+bw_t)^2+\frac{1-b^2}{2\epsilon}w_t^2]dx\\
=-\int_0^1 (Gw_t+\frac{Gw}{2\epsilon})dx.\label{emth2}
\end{multline}

Integrate both sides of (\ref{emth2}) on (0,t) with respect to time:
\begin{multline}
\frac{1}{2}\int_0^1 [(w_t+\frac{w}{2\epsilon})^2+\frac{1}{4\epsilon^2}w^2+w_x^2](x,t)dx + \int_0^t \int_0^1 [\frac{1}{2\epsilon}(w_x+bw_t)^2+\frac{1-b^2}{2\epsilon}w_t^2]dxds\\
=\frac{1}{2}\int_0^1 [(w_t+\frac{w}{2\epsilon})^2+\frac{1}{4\epsilon^2}w^2+w_x^2](x,0)dx-\int_0^t \int_0^1 (Gw_t+\frac{Gw}{2\epsilon})dxds.\label{emth3}
\end{multline}

By Cauchy-Schwarz inequality:
\begin{equation}
|\int_0^t \int_0^1 Gw_t dxds| \leq \frac{1-b^2}{4\epsilon}\int_0^t \int_0^1 w_t^2 dxds + \frac{\epsilon}{1-b^2}\int_0^t \int_0^1 G^2 dxds.
\end{equation}

Then (\ref{emth3}) becomes
\begin{multline}
\frac{1}{2}\int_0^1 [(w_t+\frac{w}{2\epsilon})^2+\frac{1}{4\epsilon^2}w^2+w_x^2](x,t)dx + \int_0^t \int_0^1 [\frac{1}{2\epsilon}(w_x+bw_t)^2+\frac{1-b^2}{4\epsilon}w_t^2]dxds\\
 \leq \frac{1}{2}\int_0^1 [(w_t+\frac{w}{2\epsilon})^2+\frac{1}{4\epsilon^2}w^2+w_x^2](x,0)dx-\int_0^t \int_0^1 \frac{Gw}{2\epsilon}dxds + \frac{\epsilon}{1-b^2} \int_0^t \int_0^1 G^2 dxds.
\end{multline}
With the initial conditions (\ref{wic}), (\ref{wtic}), we have
\begin{multline}
\frac{1}{2}\int_0^1 [(w_t+\frac{w}{2\epsilon})^2+\frac{1}{4\epsilon^2}w^2+w_x^2](x,t)dx + \int_0^t \int_0^1 [\frac{1}{2\epsilon}(w_x+bw_t)^2+\frac{1-b^2}{4\epsilon}w_t^2]dxds\\
 \leq C \int_0^1 (w_t^2+\frac{1}{\epsilon^2}w^2+w_x^2)(x, 0) dx-\int_0^t \int_0^1 \frac{Gw}{2\epsilon}dxds + \frac{\epsilon}{1-b^2} \int_0^t \int_0^1 G^2 dxds.\label{emth4}
\end{multline}

Again, by Cauchy-Schwarz inequality,
\begin{equation}
\left|\int_0^t\int_0^1 \frac{Gw}{2\epsilon} dxds\right| \leq \frac{1}{4\epsilon}\int_0^t\int_0^1(w^2+G^2)dxds .\label{gwdi} 
\end{equation}
From (\ref{gwdi}) and (\ref{emth4}), we have
\begin{align}
&\frac{1}{8\epsilon^2}\int_0^1 w^2(x,t) dx\notag\\
& \leq C\int_0^1 (w_t^2+\frac{1}{\epsilon^2}w^2 + w_x^2)(x, 0) dx + \frac{1}{4\epsilon}\int_0^t\int_0^1w^2dxds+ \frac{1}{2\epsilon}\int_0^t\int_0^1 G^2dxds\label{wbnd}
\end{align}
for small $\epsilon$. 

From the definition of $G$, we can see that, 
\begin{equation}\label{kk}\int_0^1 G^2(x,t) dx\le C\max_{[0, 1]}\sum_{i=1}^3|f^{(i)}(x)|^2. \end{equation}

Denote
\begin{equation}
F(t)=\int_0^t\int_0^1 w^2 dxds.
\end{equation}
In view of \eqref{initialestimates}, and \eqref{kk}, 
(\ref{wbnd}) becomes
\begin{align}
&F'(t)-2\epsilon F(t)\notag\\
& \leq C\epsilon^2 \int_0^1 (w_t^2+\frac{1}{\epsilon^2}w^2+w_x^2)(x, 0) dx + 4\epsilon\int_0^t\int_0^1 G^2dxds\notag\\
&\le C\epsilon^2\left(\int_0^1((g'(x))^2+b^2|f'(x)|^2)dx+ (|f'(0)|^2+|f^{''}(0)|^2)\right)\notag\\
&+C\epsilon t \max_{[0, 1]}\sum_{i=1}^3|f^{(i)}(x)|^2.
\label{wbnd1}
\end{align}
for small $\epsilon$. 

Multiply this by $e^{-2 \epsilon t}$ on both sides, we obtain,
\begin{align}
&(e^{-2\epsilon t}F(t))'\notag\\
&\leq C\epsilon^2 e^{-2\epsilon t} \left(\int_0^1((g'(x))^2+b^2|f'(x)|^2)dx+ (|f'(0)|^2+|f^{''}(0)|^2)\right)\notag\\
&+C\epsilon t e^{-2\epsilon t}\max_{[0,1]}\sum_{i=1}^3|f^{(i)}(x)|^2,\\
\end{align}
Integrate the above from $0$ to $t$ and multiply both sides by $e^{2\epsilon t}$,
\begin{align}
&e^{-2\epsilon t}F(t) \notag\\
&\leq C\epsilon(e^{-2\epsilon t}-1)\left(\int_0^1((g'(x))^2+b^2|f'(x)|^2)dx+ (|f'(0)|^2+|f^{''}(0)|^2)\right)\notag\\
&+C\epsilon\left(\frac{t e^{-2\epsilon t}}{2\epsilon} + \frac{e^{-2\epsilon t}}{4 \epsilon^2}-\frac{1}{4\epsilon^2}\right)\max_{[0,1]}\sum_{i=1}^3 |f^{(i)}(x)|^2,\notag\\
&F(t) \notag\\
&\leq C\epsilon(1-e^{2\epsilon t})\left(\int_0^1((g'(x))^2+b^2|f'(x)|^2)dx+ (|f'(0)|^2+|f^{''}(0)|^2)\right)\notag\\
&+C\epsilon\left(\frac{t}{2\epsilon} + \frac{1}{4 \epsilon^2}-\frac{e^{2\epsilon t}}{4\epsilon^2}\right)\max_{[0,1]}\sum_{i=1}^3 |f^{(i)}(x)|^2.\notag\label{wbnd4}\\
\end{align}
Do Taylor expansion for $e^{2\epsilon t}$, 
\begin{equation}
    1-e^{2\epsilon t}=O(\epsilon t),
\end{equation}
\begin{equation}
    \frac{t}{2\epsilon} + \frac{1}{4 \epsilon^2}-\frac{e^{2\epsilon t}}{4\epsilon^2} =\frac{t}{2\epsilon} + \frac{1}{4 \epsilon^2}-\frac{1+2\epsilon t+O((2\epsilon t)^2)}{4\epsilon^2} = O(t^2).
\end{equation}
Plugging the above equations into \eqref{wbnd4} , we arrive at
\begin{align}
&F(t)\notag\\
&\le C \epsilon^2 t\left(\int_0^1((g'(x))^2+b^2|f'(x)|^2)dx+ (|f'(0)|^2+|f^{''}(0)|^2)\right)\notag\\
&+C\epsilon t^2 \max_{[0, 1]}\sum_{i=1}^3|f^{(i)}(x)|^2.
\label{wbnd2}
\end{align}
This proves \eqref{wbnd3}. 
\section{Numerical Results}
We obtained the results from the Fourier solutions in Theorem~\ref{thm2.1} with $f(x)=\sin(\pi x)$, $g'(x) = -\pi \sin(\pi x)$, $\epsilon=0.01$, $a=2$ and $b = 1$. The horizontal axis in the following figures stands for $x$ and the vertical axis stands for $u$.\\
\includegraphics[width=\textwidth]{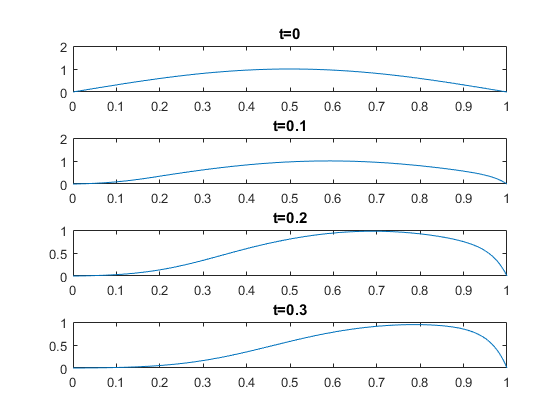}
\includegraphics[width=\textwidth]{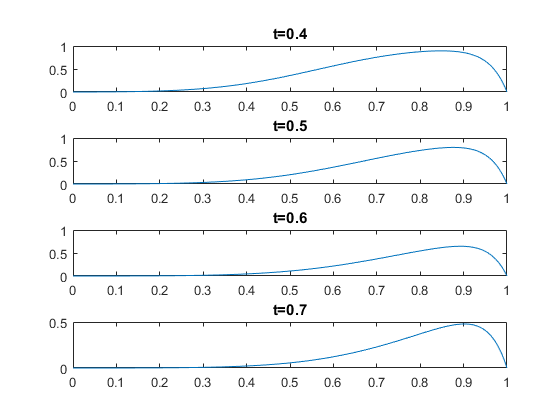}
\includegraphics[width=\textwidth]{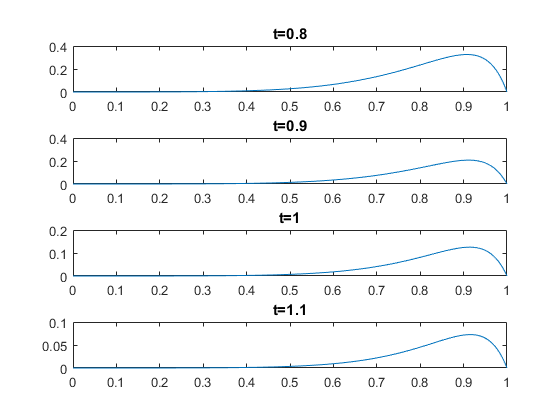}
\includegraphics[width=\textwidth]{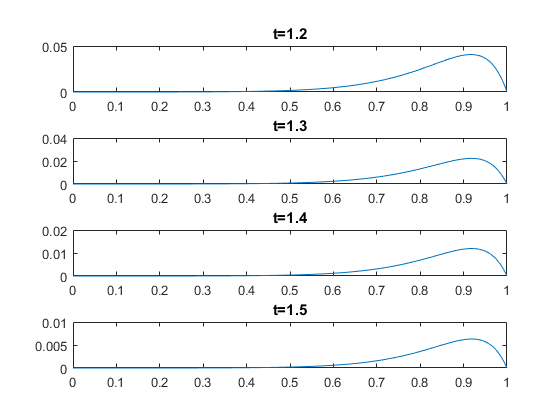}
\includegraphics[width=\textwidth]{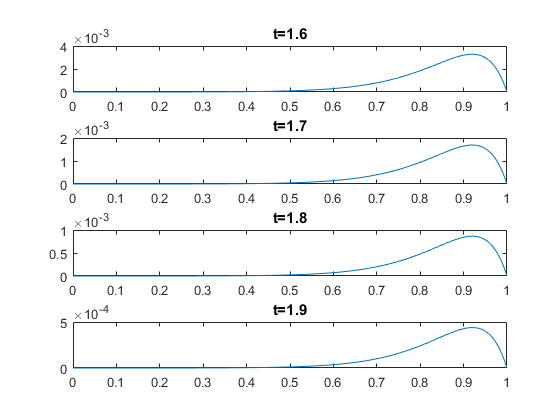}

\centerline{\bf Ackowledgements} 
The results obtained in this project were under the supervision of  Professor Tao Luo at every stage, to whom the authors are greatly grateful. 
\bibliographystyle{plain}
\bibliography{reference}
CEN Luyu: luyucen2-c@my.cityu.edu.hk \\ 
LIN Lu:   lulin22-c@my.cityu.edu.hk  \\ 
LIU Jiyuan: jiyuanliu2-c@my.cityu.edu.hk \\ 
XIAO Yujie: yujiexiao3-c@my.cityu.edu.hk

\end{document}